\documentclass[12pt,reqno]{amsart}

\usepackage[vmargin=1in,hmargin=1.2in]{geometry}
\usepackage{graphicx, latexsym, graphics}
\usepackage{amsfonts, amssymb, amsmath, amsthm, mathtools, mathrsfs}
\usepackage{algorithmic, algorithm}

\newcommand{\NN}{\mathbb N}

\newcommand{\CC}{\mathbb C}
\newcommand{\RR}{\mathbb R}
\newcommand{\ZZ}{\mathbb Z}

\newcommand{\EE}{\mathcal E}
\newcommand{\DD}{\mathcal D}
\newcommand{\SSS}{\mathcal S}

\newcommand{\Ss}{\mathbb{S}}

\newcommand{\supp}{\operatorname{supp}}
\newcommand{\Op}{\operatorname{Op}}
\newcommand{\loc}{\operatorname{loc}}
\newcommand{\comp}{\operatorname{comp}}

\theoremstyle{plain}
\newtheorem{theorem}{Theorem}[section]
\newtheorem{proposition}[theorem]{Proposition}
\newtheorem{lemma}[theorem]{Lemma}
\newtheorem{corollary}[theorem]{Corollary}

\theoremstyle{remark}
\newtheorem{remark}[theorem]{Remark}

\theoremstyle{definition}
\newtheorem{definition}[theorem]{Definition}

\allowdisplaybreaks

\numberwithin{equation}{section}

\begin{document}

\author[P. Dimovski]{Pavel Dimovski}
\thanks{The work of P. Dimovski and B. Prangoski was partially supported by the bilateral project ``Microlocal analysis and applications'' funded by the Macedonian Academy of Sciences and Arts and the Serbian Academy of Sciences and Arts.}
\address{P. Dimovski, Faculty of Technology and Metallurgy, Ss. Cyril and Methodius University in Skopje, Ruger Boskovic 16, 1000 Skopje, Macedonia}
\email{dimovski.pavel@gmail.com}

\author[S. Pilipovi\' c]{Stevan Pilipovi\' c}
\thanks{The work of S. Pilipovi\'c was partially supported by the Serbian Academy of Sciences and Arts, project, F10.}
\address{Department of Mathematics and Informatics,
University of Novi Sad, Trg Dositeja Obradovi\'{c}a 4, 21000 Novi Sad, Serbia}
\email{stevan.pilipovic@dmi.uns.ac.rs}

\author[B. Prangoski]{Bojan Prangoski}
\address{Department of Mathematics, Faculty of Mechanical
Engineering-Skopje, Ss. Cyril and Methodius University in Skopje, Karposh 2 b.b., 1000 Skopje, Macedonia}
\email{bprangoski@yahoo.com}

\title[Mikhlin multipliers which do not preserve $L^1$-, $L^{\infty}$-regularity and continuity]{On a class of Mikhlin multipliers which do not preserve $L^1$-, $L^{\infty}$-regularity and continuity}

\keywords{Fourier multipliers, Mikhlin multipliers, Fourier multipliers discontinuous on $L^1$ and $L^{\infty}$, Fourier multipliers that do not preserve continuity, Wave front set with respect to a Banach space}


\frenchspacing
\begin{abstract}
We show that every Fourier multiplier with real-valued and positively homogeneous symbol of order $0$, supported in a cone whose dual cone has a nonempty interior and such that the average of the positive part is sufficiently larger than the average of the negative part does not preserve the $L^1$- nor the $L^{\infty}$-regularity and neither the continuity. We also construct wave front sets which measure the microlocal regularity with respect to a large class of Banach spaces. As a consequence of the first part, we argue that one can never construct wave front sets that behave in a natural way and measure the microlocal $L^1$- nor $L^{\infty}$-regularity and neither the continuity.
\end{abstract}

\maketitle

\section{Introduction and preliminaries}

If $a\in \mathcal{C}^{\infty}(\RR^n)$ satisfies $\sup_{\xi\in\RR^n}(1+|\xi|)^{|\alpha|}|\partial^{\alpha} a(\xi)|<\infty$, $\alpha\in\NN^n$, both the Lizorkin-Marcinkiewicz multiplier theorem \cite{lizorkin} and the Mikhlin multiplier theorem \cite{mikhlin} verify that $a$ is a Fourier multiplier for $L^p(\RR^n)$, $1<p<\infty$. As standard, we call the smooth functions which satisfy these bounds Mikhlin multipliers.\\
\indent More generally, let $\Op(a)$ be the pseudo-differential operator (from now, often abbreviated as $\Psi$DO) with symbol $a\in S^r_{\rho,\delta}(\RR^{2n})$ defined by
\begin{equation*}
\Op(a)\varphi(x):=\frac{1}{(2\pi)^n}\int_{\RR^n}e^{ix\xi}a(x,\xi)\mathcal{F}\varphi(\xi) d\xi,\quad \varphi\in\SSS(\RR^n).
\end{equation*}
Here $S^r_{\rho,\delta}(\RR^{2n})$, $r\in\RR$, $0\leq \delta\leq\rho\leq 1$, $\delta<1$, is the Fr\'echet space of global symbols on $\RR^{2n}$ \cite{hor2} consisting of all $a\in\mathcal{C}^{\infty}(\RR^{2n})$ which satisfy $\sup_{x,\xi\in\RR^n}\langle\xi\rangle^{-r+\rho|\alpha|-\delta|\beta|}|\partial^{\alpha}_{\xi}\partial^{\beta}_x a(x,\xi)|<\infty$, $\alpha,\beta\in\NN^n$. The $\Psi$DO $\Op(a)$ is continuous on $\SSS(\RR^n)$ and it extends to a continuous operator on $\SSS'(\RR^n)$. Furthermore, if $a\in S^0_{1,\delta}(\RR^n)$ then $\Op(a)$ is continuous on $L^p(\RR^n)$, $1<p<\infty$; see \cite[Chapter 7, p. 323]{stein} (see also \cite{illner,taylor}). The $L^p$-continuity, $1<p<\infty$, $p\neq 2$, has also been established even in the case when $\rho<1$ but then the order $r$ is strictly less than $0$ and $r$, $p$, $\rho$ and $\delta$ have to satisfy certain inequality which is known to be optimal; see \cite{fefferman}. We refer to \cite{fefferman} also for results concerning $H^1(\RR^n)\rightarrow L^1(\RR^n)$ and $L^{\infty}(\RR^n)\rightarrow BMO(\RR^n)$ continuity of $\Psi$DOs. We also point out \cite{a-h,g-z,hor1,rodino,wan} and the references therein for the continuity properties on $L^p$, $BMO$ and the Hardy spaces of pseudo-differential operators with symbols in the pathological cases when $1\geq \delta>\rho$ or $\delta=\rho=1$. When the dimension $n$ is $1$, it is folklore that a Mikhlin multiplier (i.e. a $\Psi$DO with symbol independent of $x$) which equals $\operatorname{sgn}(\xi)$ in a neighbourhood of $\pm\infty$ gives an example of a $\Psi$DO with symbol in $S^0_{1,0}(\RR^2)$ which is not continuous on $L^1(\RR)$ nor on $L^{\infty}(\RR)$. The reason is that the Fourier side of such function is not a bounded measure and, by the theorem of H\"ormander on translation invariant operators \cite[Theorem 1.4]{hor-p1}, the spaces of continuous operators on $L^1(\RR)$ and on $L^{\infty}(\RR)$ are exactly the space of bounded Radon measures on $\RR$. This example can be generalised to $n>1$ dimensions by taking an $n$-fold tensor product of the one-dimensional example, however, the resulting function is not a Mikhlin multiplier: it does not belong to $S^0_{\rho,0}(\RR^{2n})$ for any $\rho>0$. In this article we show a more general result: if the Fourier multiplier is positively homogeneous of order $0$ away from the origin (consequently, it is a Mikhlin multiplier), has a nonnegative (or nonpositive) real or imaginary part with support in a closed cone $V$ whose dual cone has a nonempty interior then it is not a Fourier multiplier neither for $L^1(\RR^n)$ nor for $L^{\infty}(\RR^n)$. In fact, the real (or imaginary) part does not have to be nonnegative, only the average of the positive part to be bigger than a constant multiple of the average of the negative part where the constant is a measure of the angular distance between $V$ and its dual cone; see \eqref{def-ofk-for-funconsp} and \eqref{con-for-exi-offournotmeas} below. Moreover, we show something stronger: these Mikhlin multipliers are never continuous operators $L^1_{\comp}(\RR^n)\rightarrow L^1_{\loc}(\RR^n)$, neither $L^{\infty}_{\comp}(\RR^n)\rightarrow L^{\infty}_{\loc}(\RR^n)$, nor $\mathcal{K}(\RR^n)\rightarrow \mathcal{C}(\RR^n)$; here $\mathcal{K}(\RR^n)$ is the space of continuous functions with compact support equipped with its standard strict $(LB)$-space topology and, as usual, $\mathcal{C}(\RR^n)$ is the Fr\'echet space of continuous functions on $\RR^n$. This shows that the reason for the discontinuity is not the growth but the fact that they destroy the local $L^1$-regularity, $L^{\infty}$-regularity and continuity. The last fact that these operators are not continuous as maps $\mathcal{K}(\RR^n)\rightarrow \mathcal{C}(\RR^n)$ is interesting in itself since it is known that all $\Psi$DOs with symbols in $S^0_{1,0}(\RR^{2n})$ are continuous on the H\"older classes of order $\rho>0$ when $\rho$ is not an integer; see \cite[Theorem 8.6.14, p. 209]{hor1} and the comments after \cite[Definition 8.6.4, p. 203]{hor1}.\\
\indent In the last section, we construct wave front sets that measure the microlocal regularity of distributions with respect to a general Banach space $E$. By choosing $E$ appropriately, one recovers most of the wave front sets that appear in the literature: the Sobolev wave front set \cite{hor1}, the Besov wave front set \cite{dap-r-scl}, the wave front sets considered in \cite{cor-joh-tof1}, etc. It is important to point out that we do not assume neither that $E$ nor that the Fourier side of $E$ is solid. Finally, by applying the main result from the first part, we argue that one can never construct wave front sets that behave in a natural way and measure the microlocal $L^1$-regularity, $L^{\infty}$-regularity, or continuity of distributions.\\
\indent We end this section by fixing the notations in the article. As standard, we denote $\langle x\rangle:=(1+|x|^2)^{1/2}$, $x\in\RR^n$. Given $r>0$ and $x\in\RR^n$, set $B_r(x):=\{y\in\RR^n\,|\, |x-y|\leq r\}$. Given a measurable set $A\subseteq \RR^n$, $\mathbf{1}_A$ stands for the indicator function of $A$. We fix the constants in the Fourier transform as $\mathcal{F}f(\xi):=\int_{\RR^n}e^{-ix\xi}f(x)dx$, $f\in L^1(\RR^n)$. We write $\check{f}(x):=f(-x)$ for reflection about the origin. For a cone $V$ in $\RR^n$, we denote by $V^*$ its dual cone, i.e. $V^*:=\{x\in\RR^n\,|\, xy\geq0,\, \forall y\in V\}$; $V^*$ is always closed and $V^*=(\overline{V})^*$. The set $L\subseteq O\times (\RR^n\backslash\{0\})$, $O$ open in $\RR^n$, is said to be conic if it satisfies: $(x,\xi)\in L$ implies $(x,\lambda\xi)\in L$, $\forall \lambda>0$. When $a\in S^0_{1,0}(\RR^{2n})$ does not depend on $x$, i.e. $a(x,\xi)=a(\xi)$, we write $a(D)$ instead of $\Op(a)$ and we point out that $a(D)\varphi=\mathcal{F}^{-1}(a)*\varphi$, $\varphi\in\SSS(\RR^n)$. Given a compact set $K\subseteq \RR^n$, we denote by $\mathcal{K}_K(\RR^n)$ the Banach space of continuous functions on $\RR^n$ with support in $K$ while $\DD_K(\RR^n)$ stands for the Fr\'echet space of smooth functions on $\RR^n$ supported by $K$. The strong dual of $\mathcal{K}(\RR^n)$ is denoted by $\mathcal{K}'(\RR^n)$ and is the space of Radon measures on $\RR^n$ (see \cite[Chapter 3]{bourbaki}).\\
\indent Given two locally convex spaces $X$ and $Y$ (from now, always abbreviated as l.c.s.), $\mathcal{L}(X,Y)$ stands for the space of continuous linear operators from $X$ into $Y$ and we denote by $\mathcal{L}_b(X,Y)$ this space equipped with the strong operator topology. When $X=Y$, we abbreviate notations and write $\mathcal{L}(X)$ and $\mathcal{L}_b(X)$.

\section{The main result}

The following lemma is the key ingredient for the main result.

\begin{lemma}\label{lem-for-con-exampl}
Let $V\subseteq \RR^n$ be a closed cone such that $V\backslash\{0\}\neq \emptyset$ and $V':=\operatorname{int}V^*\neq \emptyset$. Let $\varphi_1\in L^{\infty}(\Ss^{n-1})$ be a real-valued function satisfying $\supp\varphi_1\subseteq \Ss^{n-1}\cap V$ and set $\varphi_{1,+}:=\max\{\varphi_1,0\}$ and $\varphi_{1,-}:=\max\{-\varphi_1,0\}$. Define
\begin{equation}\label{def-ofk-for-funconsp}
\kappa_0:=\sup_{\omega'\in \mathbf{S}^{n-1}\cap V'}(\inf_{\omega\in\supp\varphi_{1,-}}\omega\omega').
\end{equation}
If $\supp\varphi_{1,-}\neq \emptyset$ then $\kappa_0\in(0,1]$, otherwise $\kappa_0=\infty$. Assume that
\begin{equation}\label{con-for-exi-offournotmeas}
\int_{\mathbb{S}^{n-1}}\varphi_{1,+}(\omega)d\omega>\kappa_0^{-n}\int_{\mathbb{S}^{n-1}}\varphi_{1,-}(\omega)d\omega.
\end{equation}
Then for any $r>0$ the function $\varphi:\RR^n\rightarrow\RR$, $\varphi(x):=\varphi_1(x/|x|)\mathbf{1}_{(r,\infty)}(|x|)$, belongs to $L^{\infty}(\RR^n)$ but $\mathcal{F}\varphi\not\in\mathcal{K}'(\RR^n)$.
\end{lemma}

\begin{remark}
Assuming $\supp\varphi_{1,-}\neq \emptyset$, the geometric meaning of $\kappa_0$ is the following. For fixed $\omega'\in V'$, $\inf_{\omega\in\supp\varphi_{1,-}}\omega\omega'$ is the cosine of the largest angle between $\omega'$ and the points in $\supp\varphi_{1,-}$, i.e. it is the cosine of half of the aperture of the smallest circular cone with axis $\omega'$ that contains $\supp\varphi_{1,-}$. The parameter $\kappa_0$ is the cosine of half of the aperture of the smallest circular cone with axis in $\overline{V'}$ that contains $\supp\varphi_{1,-}$ (this cone may not be unique): the bigger $\kappa_0$ is the smaller the aperture is.
\end{remark}

\begin{proof}[Proof of Lemma \ref{lem-for-con-exampl}] The proof of the facts on $\kappa_0$ is straightforward an we omit it. For the rest of the claims, assume first that $n\geq 2$. Clearly $\varphi\in L^{\infty}(\RR^n)$ (its measurability follows from the fact that preimages of nullsets by the map $\RR^n\backslash\{0\}\rightarrow \mathbb{S}^{n-1}$, $x\mapsto x/|x|$, are nullsets). There is $\omega'_0\in\mathbb{S}^{n-1}\cap V'$ such that \eqref{con-for-exi-offournotmeas} is satisfied with $\inf_{\omega\in\supp\varphi_{1,-}}\omega\omega'_0$ in place of $\kappa_0$. Pick an open cone $V''$ containing $\omega'_0$ and satisfying $\overline{V''}\subseteq V'\cup\{0\}$ such that \eqref{con-for-exi-offournotmeas} holds true with
$$
\kappa'_0:=\left\{
\begin{array}{l}
  \inf_{\omega'\in \mathbf{S}^{n-1}\cap \overline{V''}}\inf_{\omega\in\supp\varphi_{1,-}}\omega\omega',\,\,\, \mbox{if}\,\,\, \supp\varphi_{1,-}\neq\emptyset,\\
  1,\,\,\, \mbox{if}\,\,\, \supp\varphi_{1,-}=\emptyset,
\end{array}
\right.
$$
in place of $\kappa_0$. Without loss of generality, we can assume $\omega'_0=(1,0,\ldots,0)$ for otherwise we can employ rotation to reduce the problem to this. The compactness of $\Ss^{n-1}$ implies that there is $\varepsilon_0>0$ such that $\Ss^{n-1}\cap V\subseteq \{\omega=(\omega_1,\ldots,\omega_n)\in\Ss^{n-1}\,|\, \omega_1\geq \varepsilon_0\}$. Our goal is to calculate $\mathcal{F}\varphi$. For $\chi\in\SSS(\RR^n)$, we employ spherical coordinates to infer
\begin{align*}
\langle\mathcal{F}\varphi,\chi\rangle&=\int_{\Ss^{n-1}\cap V} \int_r^{\infty}\varphi_1(\omega)\rho^{n-1}\mathcal{F}\chi(\rho\omega)d\rho d\omega\\
&=\int_{\Ss^{n-1}\cap V}\int_r^{\infty}\varphi_1(\omega)\omega_1^{-n-1}\rho^{-2}\mathcal{F}(D^{n+1}_1\chi)(\rho\omega)d\rho d\omega.
\end{align*}
Setting $\phi_1(\omega):=\varphi_1(\omega)/\omega_1^{n+1}$ and $\psi(x):=\phi_1(x/|x|)\mathbf{1}_{(r,\infty)}(|x|)/|x|^{n+1}$, we deduce that $\phi_1\in L^{\infty}(\Ss^{n-1})$ with $\phi_1=0$ a.e. on $\Ss^{n-1}\backslash V$, $\psi\in L^1(\RR^n)\cap L^{\infty}(\RR^n)$ and $\mathcal{F}\varphi=(-1)^{n+1}D^{n+1}_1\mathcal{F}\psi$. Notice that
$$
\mathcal{F}\psi(\xi)=\int_{\Ss^{n-1}\cap V}\int_r^{\infty}\phi_1(w)e^{-i\rho\omega\xi}\rho^{-2}d\rho d\omega
$$
and consequently
$$
\langle\mathcal{F}\varphi,\chi\rangle=\int_{\Ss^{n-1}\cap V}\int_{\RR^n}\int_r^{\infty}\phi_1(\omega)e^{-i\rho\omega\xi}\rho^{-2}D^{n+1}_{\xi_1}\chi(\xi)d\rho d\xi d\omega,\quad \chi\in\SSS(\RR^n).
$$
Set $f(\lambda):=\int_{\lambda}^{\infty} e^{-i\rho}\rho^{-2}d\rho$, $\lambda>0$; clearly $f\in\mathcal{C}^{\infty}(0,\infty)$. For $\chi\in\DD(V')$, we have
\begin{align*}
\langle\mathcal{F}\varphi,\chi\rangle&=\int_{\Ss^{n-1}\cap V}\int_{V'}\phi_1(\omega)(\omega\xi)f(r\omega\xi)D^{n+1}_{\xi_1}\chi(\xi) d\xi d\omega\\
&=(-1)^{n+1}\int_{\Ss^{n-1}\cap V}\int_{V'}\varphi_1(\omega)(\omega\xi)r^{n+1}f^{(n+1)}(r\omega\xi)\chi(\xi) d\xi d\omega\\
&{}\quad+(-1)^{n+1}(n+1) \int_{\Ss^{n-1}\cap V}\int_{V'}\varphi_1(\omega)r^nf^{(n)}(r\omega\xi)\chi(\xi) d\xi d\omega\\
&=\sum_{k=0}^n{n\choose k}i^{n-k}(k+1)!\int_{\Ss^{n-1}\cap V}\int_{V'}\varphi_1(\omega)r^{n-k-1}e^{-ir\omega\xi}(\omega\xi)^{-k-1}\chi(\xi) d\xi d\omega\\
&{}\quad-(n+1) \sum_{k=0}^{n-1}{{n-1}\choose k}i^{n-k-1}(k+1)!\int_{\Ss^{n-1}\cap V}\int_{V'}\varphi_1(\omega)r^{n-k-2}e^{-ir\omega\xi}(\omega\xi)^{-k-2}\chi(\xi) d\xi d\omega\\
&=\sum_{k=0}^{n-2}{n\choose k}i^{n-k}(k+1)!\int_{\Ss^{n-1}\cap V}\int_{V'}\varphi_1(\omega)r^{n-k-1}e^{-ir\omega\xi}(\omega\xi)^{-k-1}\chi(\xi) d\xi d\omega\\
&{}\quad-(n+1) \sum_{k=0}^{n-3}{{n-1}\choose k}i^{n-k-1}(k+1)!\int_{\Ss^{n-1}\cap V}\int_{V'}\varphi_1(\omega)r^{n-k-2}e^{-ir\omega\xi}(\omega\xi)^{-k-2}\chi(\xi) d\xi d\omega\\
&{}\quad+i(n-1)!\int_{\Ss^{n-1}\cap V}\int_{V'}\varphi_1(\omega)e^{-ir\omega\xi}(\omega\xi)^{-n}\chi(\xi) d\xi d\omega;
\end{align*}
of course, when $n=2$ the second sum is vacuous. Assume that $\mathcal{F}\varphi\in \mathcal{K}'(\RR^n)$. Then, there is $C>1$ such that $|\langle\mathcal{F}\varphi,\chi\rangle|\leq C\|\chi\|_{L^{\infty}(\RR^n)}$, for all $\chi\in\DD(\RR^n)$ with $\supp\chi\subseteq B_1(0)$.
The compactness of $\Ss^{n-1}$ implies that there is $\varepsilon'>0$ such that $\omega\xi\geq \varepsilon'|\xi|$, $\omega\in \Ss^{n-1}\cap V$, $\xi\in\overline{V''}$. When $\chi\in\DD(V'')$ with $\supp\chi\subseteq B_1(0)$, all terms in the first and second sum satisfy the following estimate
\begin{multline*}
\left|\int_{\Ss^{n-1}\cap V}\int_{V'}\varphi_1(\omega)e^{-ir\omega\xi}(\omega\xi)^{-k}\chi(\xi) d\xi d\omega\right|\\
\leq \varepsilon'^{-k}\|\chi\|_{L^{\infty}(\RR^n)}\int_{\Ss^{n-1}\cap V}\int_{B_1(0)}|\varphi_1(\omega)||\xi|^{-k}d\xi d\omega=C'_k\|\chi\|_{L^{\infty}(\RR^n)},
\end{multline*}
when $1\leq k\leq n-1$. Consequently, there is $C''>1$ such that
$$
\left|\int_{\Ss^{n-1}\cap V}\int_{V'}\varphi_1(\omega)e^{-ir\omega\xi}(\omega\xi)^{-n}\chi(\xi) d\xi d\omega\right|\leq C''\|\chi\|_{L^{\infty}(\RR^n)},\,\, \chi\in \DD(V''),\, \supp\chi\subseteq B_1(0).
$$
Since \eqref{con-for-exi-offournotmeas} holds true with $\kappa_0'$, there is $s>1$ such that
\begin{equation}\label{ine-for-int-ofvarpgunc}
\cos(1/s)\|\varphi_{1,+}\|_{L^1(\Ss^{n-1})}>\kappa_0'^{-n}\|\varphi_{1,-}\|_{L^1(\Ss^{n-1})}.
\end{equation}
For each $l\in\ZZ_+$, $l\geq 2s+1$, pick $\theta_l\in\DD(1/(2l(r+1)),1/(s(r+1)))$ such that $0\leq \theta_l\leq 1$ and $\theta_l=1$ on $[1/(l(r+1)),1/(2s(r+1))]$. Choose $\widetilde{\chi}\in\DD(V'')$ which satisfies $0\leq\widetilde{\chi}\leq 1$ and $\widetilde{\chi}(1,0,\ldots,0)=1$ and define $\chi_l(\xi):=\widetilde{\chi}(\xi/|\xi|)\theta_l(|\xi|)$. Clearly $\chi_l\in\DD'(V'')$ with $\supp\chi_l\subseteq B_1(0)$. The above implies that
\begin{align*}
C''&\geq \int_{\Ss^{n-1}\cap V}\int_{V''}\varphi_1(\omega)\cos(r\omega\xi)(\omega\xi)^{-n}\chi_l(\xi) d\xi d\omega\\
&=\int_{\Ss^{n-1}\cap V}\int_{\Ss^{n-1}\cap V''} \int_0^{1/(s(r+1))} \varphi_{1,+}(\omega)\cos(r\rho'\omega\omega')(\omega\omega')^{-n}\rho'^{-1}\widetilde{\chi}(\omega')\theta_l(\rho') d\rho'd\omega' d\omega\\
&{}\quad -\int_{\Ss^{n-1}\cap V}\int_{\Ss^{n-1}\cap V''} \int_0^{1/(s(r+1))} \varphi_{1,-}(\omega)\cos(r\rho'\omega\omega')(\omega\omega')^{-n}\rho'^{-1}\widetilde{\chi}(\omega')\theta_l(\rho') d\rho'd\omega' d\omega\\
&\geq \cos(1/s)\|\varphi_{1,+}\|_{L^1(\Ss^{n-1})}\|\widetilde{\chi}\|_{L^1(\Ss^{n-1})} \int_{1/(l(r+1))}^{1/(2s(r+1))} \frac{d\rho'}{\rho'}\\
&{}\quad -\kappa_0'^{-n}\|\varphi_{1,-}\|_{L^1(\Ss^{n-1})}\|\widetilde{\chi}\|_{L^1(\Ss^{n-1})} \int_{1/(2l(r+1))}^{1/(s(r+1))} \frac{d\rho'}{\rho'}\\
&= \|\widetilde{\chi}\|_{L^1(\Ss^{n-1})}\left(\cos(1/s)\|\varphi_{1,+}\|_{L^1(\Ss^{n-1})}\ln(l/(2s))-\kappa_0'^{-n} \|\varphi_{1,-}\|_{L^1(\Ss^{n-1})}\ln(2l/s)\right),
\end{align*}
for all $l\in\ZZ_+$, $l\geq 2s+1$. This is a contradiction since the very last term tends to $\infty$ as $l\rightarrow\infty$ in view of \eqref{ine-for-int-ofvarpgunc} and the proof is complete.\\
\indent It remains to show the claim when $n=1$. Without loss in generality we can assume that $V=[0,\infty)$ and thus $\varphi(x)=c\mathbf{1}_{(r,\infty)}(x)$, $x\in\RR$, with some $c>0$. Since $\mathcal{F}(\operatorname{sgn}(\cdot))=-2i\operatorname{PV}(1/\cdot)$ \footnote{$\operatorname{PV}(1/\xi)$ is the principle value of $1/\xi$ defined by $\langle\operatorname{PV}(1/\cdot),\phi\rangle=\lim_{\varepsilon\rightarrow0^+}\int_{|\xi|\geq \varepsilon}\frac{\phi(\xi)}{\xi}d\xi$, $\phi\in\DD(\RR)$.} and $\operatorname{sgn}(x-r)=\frac{2}{c}\varphi(x)-1$, $x\in\RR\backslash\{r\}$, we infer $\mathcal{F}\varphi=-ice^{-ir\,\cdot}\operatorname{PV}(1/\cdot)+c\pi\delta$. We are going to show that $\operatorname{PV}(1/\cdot)\not\in\mathcal{K}'(\RR)$ which will imply that $\mathcal{F}\varphi\not\in\mathcal{K}'(\RR)$. Assume the contrary. Then there is $C>0$ such that
$$
|\langle \operatorname{PV}(1/\cdot),\chi\rangle|\leq C\|\chi\|_{L^{\infty}(\RR)},\quad \chi\in \DD_{[-1,1]}(\RR).
$$
For each $k\in\ZZ_+$, $k\geq 3$, pick $\chi_k\in\DD(0,1)$ such that $0\leq \chi_k\leq 1$ and $\chi_k=1$ on $[1/k,1/2]$. Notice that
$$
C=C\|\chi_k\|_{L^{\infty}(\RR)}\geq|\langle \operatorname{PV}(1/\cdot),\chi_k\rangle|=\int_0^1\frac{\chi_k(x)}{x}dx\geq \int_{1/k}^{1/2}\frac{dx}{x}=\ln(k/2),\quad k\geq 3,
$$
which is a contradiction and the proof is complete.
\end{proof}

\begin{remark}\label{rem-for-non-negfunconfs}
Any nonnegative $\varphi_1\in L^{\infty}(\Ss^{n-1})\backslash\{0\}$ satisfies \eqref{con-for-exi-offournotmeas} and the lemma is applicable to it.
\end{remark}

\begin{remark}
Let $V$ and $V'$ be as in the lemma and set
$$
\kappa_V:=\sup_{\omega'\in \Ss^{n-1}\cap V'}(\inf_{\omega\in \Ss^{n-1}\cap V}\omega\omega').
$$
Then $\kappa_V\in(0,1]$. If $\varphi_1\in L^{\infty}(\Ss^{n-1})$ is real-valued and satisfies $\supp\varphi_1\subseteq \Ss^{n-1}\cap V$ and
\begin{equation}\label{con-for-exi-offournotmeas11}
\int_{\mathbb{S}^{n-1}}\varphi_{1,+}(\omega)d\omega>\kappa_V^{-n}\int_{\mathbb{S}^{n-1}}\varphi_{1,-}(\omega)d\omega,
\end{equation}
then it satisfies \eqref{con-for-exi-offournotmeas} and the claim in the lemma is applicable to this $\varphi_1$.
\end{remark}

\begin{definition}
The function $\psi:\RR^n\rightarrow \CC$ is said to be positively homogeneous of order $0$ outside of $B_R(0)$ for some $R>0$ if $\psi(\lambda\xi)=\psi(\xi)$, $|\xi|> R$, $\lambda>1$. We abbreviate it as positively homogeneous of order $0$ when the number $R$ is not important.
\end{definition}

\begin{remark}
If $\psi\in\mathcal{C}^{\infty}(\RR^n)$ is positively homogeneous of order $0$, then the function $\RR^{2n}\rightarrow \CC$, $(x,\xi)\mapsto \psi(\xi)$, belongs to $S^0_{1,0}(\RR^{2n})$, i.e. it is a Mikhlin multiplier.
\end{remark}

\begin{proposition}\label{pro-for-not-multfunradcon}
Let $V\subseteq \RR^n$ be a closed cone such that $V\backslash\{0\}\neq \emptyset$ and $V':=\operatorname{int}V^*\neq \emptyset$. Let $\psi\in L^{\infty}(\RR^n)$ be such that $\operatorname{Re}\psi$ is positively homogeneous of order $0$ and $\supp(\operatorname{Re}\psi)\subseteq V$. Define $\varphi_1:\Ss^{n-1}\rightarrow \RR$, $\varphi_1(\omega):=\lim_{r\rightarrow \infty}\operatorname{Re}\psi(r\omega)$. Then $\varphi_1\in L^{\infty}(\Ss^{n-1})$ and if $\varphi_1$ or $-\varphi_1$ satisfies \eqref{con-for-exi-offournotmeas} then $\mathcal{F}\psi\not\in\mathcal{K}'(\RR^n)$.\\
\indent If this is satisfied for $\operatorname{Im}\psi$ instead of $\operatorname{Re}\psi$, then again $\mathcal{F}\psi\not\in\mathcal{K}'(\RR^n)$.
\end{proposition}

\begin{proof} We only show the proposition when $\operatorname{Re}\psi$ satisfies the assumptions as the proof when $\operatorname{Im}\psi$ satisfies the assumptions is analogous. Let $R>0$ be large enough such that $\psi(\lambda\xi)=\psi(\xi)$, $|\xi|> R$, $\lambda>1$. Then $\varphi_1(\omega)=\operatorname{Re}\psi((R+1)\omega)$, $\omega\in\Ss^{n-1}$, and the measurability of $\varphi_1$ follows (essentially) by the definition of the spherical Lebesgue measure; it is straightforward to prove that $\varphi_1\in L^{\infty}(\Ss^{n-1})$ (by contradiction). Without loss in generality, we can assume that $\varphi_1$ satisfies \eqref{con-for-exi-offournotmeas}. We apply Lemma \ref{lem-for-con-exampl} for this $\varphi_1$ to deduce that $\mathcal{F}\varphi\not\in\mathcal{K}'(\RR^n)$ with $\varphi(x):=\varphi_1(x/|x|)\mathbf{1}_{(R,\infty)}(|x|)=\operatorname{Re}\psi(x)\mathbf{1}_{(R,\infty)}(|x|)$. Since $x\mapsto \operatorname{Re}\psi(x)(1-\mathbf{1}_{(R,\infty)}(|x|))$ has compact support, its Fourier transform is smooth and hence in $\mathcal{K}'(\RR^n)$. We deduce that $\mathcal{F}(\operatorname{Re}\psi)\not\in\mathcal{K}'(\RR^n)$. If $\mathcal{F}\psi\in\mathcal{K}'(\RR^n)$ then $\mathcal{F}\overline{\psi}=(\overline{\mathcal{F}\psi})\check{}\in\mathcal{K}'(\RR^n)$ and thus $\mathcal{F}(\operatorname{Re}\psi)=(\mathcal{F}\psi+\mathcal{F}\overline{\psi})/2\in\mathcal{K}'(\RR^n)$ which is a contradiction and the proof is complete.
\end{proof}

\begin{remark}
If the function $\varphi_1$ in the proposition is nonnegative or nonpositive and $\varphi_1\neq 0$ on a set with positive measure, then $\mathcal{F}\psi\not\in\mathcal{K}'(\RR^n)$ (cf. Remark \ref{rem-for-non-negfunconfs}).
\end{remark}

Next, we show a local variant of the H\"ormander result \cite[Theorem 1.4]{hor-p1} on the characterisations of translation invariant operators on $L^1$ and $L^{\infty}$; recall that an operator $A:X\rightarrow Y$ between the translation invariant distribution spaces $X$ and $Y$ on $\RR^n$ is said to be translation invariant if it commutes with all translations, i.e. $AT_x=T_xA$, $x\in\RR^n$, where $T_x$ is the translation by $x$ defined as $T_xf:=f(\cdot-x)$. Recall that the convolution $*:\EE'(\RR^n)\times\DD'(\RR^n)\rightarrow \DD'(\RR^n)$ is well-defined and hypocontinuous. If $\mu\in\mathcal{K}'(\RR^n)$ is positive (i.e. $\langle \mu,\varphi\rangle\geq 0$ for all nonnegative $\varphi\in\mathcal{K}(\RR^n)$) and $f\in L^1_{\comp}(\RR^n)$ then it is straightforward to check that the distribution $f*\mu$ is in fact in $L^1_{\loc}(\RR^n)$ and
$$
f*\mu(x)=\int_{\RR^n}f(x-y)d\mu(y),\quad \mbox{a.a.}\,\, x\in\RR^n;
$$
(the Fubini theorem implies that $f*\mu$ is measurable and $f*\mu\in L^1_{\loc}(\RR^n)$; of course, one takes any Borel measurable representative of $f$). For general $\mu\in\mathcal{K}'(\RR^n)$, one writes
\begin{equation}\label{dec-ofm-int-mindofrealanim}
\mu=\mu_{\operatorname{Re},+}-\mu_{\operatorname{Re},-}+i(\mu_{\operatorname{Im},+}-\mu_{\operatorname{Im},-}),
\end{equation}
where $\mu_{\operatorname{Re},+}$, $\mu_{\operatorname{Re},-}$, $\mu_{\operatorname{Im},+}$ and $\mu_{\operatorname{Im},-}$ are the unique positive Radon measures such that $\mu_{\operatorname{Re},+}-\mu_{\operatorname{Re},-}$ and $\mu_{\operatorname{Im},+}-\mu_{\operatorname{Im},-}$ are the minimal decompositions of the real and imaginary part of $\mu$ respectively (see \cite[Section 4.3]{edw-bo}), and $f*\mu$ can be given as a sum of four integrals as above since
$
f*\mu=f*\mu_{\operatorname{Re},+}-f*\mu_{\operatorname{Re},-}+i(f*\mu_{\operatorname{Im},+}-f*\mu_{\operatorname{Im},-}).
$
It is straightforward to check that for each $\mu\in \mathcal{K}'(\RR^n)$, the map $f\mapsto f*\mu$ is well-defined and continuous as a map
$$
L^1_{\comp}(\RR^n)\rightarrow L^1_{\loc}(\RR^n),\quad L^{\infty}_{\comp}(\RR^n)\rightarrow L^{\infty}_{\loc}(\RR^n)\quad \mbox{and}\quad \mathcal{K}(\RR^n)\rightarrow \mathcal{C}(\RR^n).
$$
The local variant of the H\"ormander theorem \cite[Theorem 1.4]{hor-p1} is the following.

\begin{proposition}\label{spa-ofm-for-lolicont}
${}$
\begin{itemize}
\item[$(a)$] The map $\mathcal{K}'(\RR^n)\rightarrow \mathcal{L}_b(L^1_{\comp}(\RR^n), L^1_{\loc}(\RR^n))$, $\mu\mapsto (f\mapsto f*\mu)$, is a topological imbedding whose image is the space of continuous operators $L^1_{\comp}(\RR^n)\rightarrow L^1_{\loc}(\RR^n)$ that commute with all translations.
\item[$(b)$] The map $\mathcal{K}'(\RR^n)\rightarrow \mathcal{L}_b(\mathcal{K}(\RR^n), \mathcal{C}(\RR^n))$, $\mu\mapsto (\varphi\mapsto \varphi*\mu)$, is a topological imbedding whose image is the space of continuous operators $\mathcal{K}(\RR^n)\rightarrow\mathcal{C}(\RR^n)$ that commute with all translations.
\item[$(c)$] The map $\mathcal{K}'(\RR^n)\rightarrow \mathcal{L}_b(L^{\infty}_{\comp}(\RR^n), L^{\infty}_{\loc}(\RR^n))$, $\mu\mapsto (f\mapsto f*\mu)$, is a topological imbedding. If $Q:L^{\infty}_{\comp}(\RR^n)\rightarrow L^{\infty}_{\loc}(\RR^n)$ is continuous and commutes with all translations, then there is $\mu\in\mathcal{K}'(\RR^n)$ such that $Q\varphi=\varphi*\mu$, $\varphi\in\mathcal{K}(\RR^n)$.
\end{itemize}
\end{proposition}

\begin{proof} The maps in $(a)$, $(b)$ and $(c)$ are well-defined in view of the above comments. We only show the continuity for the map in $(a)$ as the proofs of the continuity for the rest are analogous. Let $B$ be a bounded subset of $L^1_{\comp}(\RR^n)$ and $K'$ a compact set in $\RR^n$. We need to bound $\sup_{f\in B}\|f*\mu\|_{L^1(K')}$ by a seminorm of $\mu$ in $\mathcal{K}'(\RR^n)$. There is a compact set $K\subseteq \RR^n$ such that $B$ is a bounded subset of $L^1_K(\RR^n)$, where the latter is the closed subspace of $L^1(\RR^n)$ consisting of the elements supported by $K$. Pick $\varphi\in\mathcal{K}(\RR^n)$ satisfying $0\leq \varphi\leq 1$ and $\varphi=1$ on $K'-K$. Write $\mu$ as in \eqref{dec-ofm-int-mindofrealanim} and denote by $\mu_{\operatorname{Re}}$ and $\mu_{\operatorname{Im}}$ the real and imaginary parts of $\mu$. For $f\in B$, we estimate as follows
\begin{align*}
\|f*\mu\|_{L^1(K')}&\leq \int_{K'-K}\int_{K'}|f(x-y)|dxd\mu_{\operatorname{Re},+}(y)+ \int_{K'-K}\int_{K'}|f(x-y)|dxd\mu_{\operatorname{Re},-}(y)\\
&{}\quad + \int_{K'-K}\int_{K'}|f(x-y)|dxd\mu_{\operatorname{Im},+}(y) + \int_{K'-K}\int_{K'}|f(x-y)|dxd\mu_{\operatorname{Im},-}(y)\\
&\leq\|f\|_{L^1(\RR^n)}(\langle\mu_{\operatorname{Re},+},\varphi\rangle+\langle\mu_{\operatorname{Re},-},\varphi\rangle+ \langle\mu_{\operatorname{Im},+},\varphi\rangle+\langle\mu_{\operatorname{Im},-},\varphi\rangle).
\end{align*}
Since $\langle\mu_{\operatorname{Re},+},\varphi\rangle=\sup\{\langle \mu_{\operatorname{Re}},\psi\rangle\,|\, \psi\in\mathcal{K}(\RR^n),\,\, 0\leq\psi\leq \varphi\}$ (see the proof of \cite[Theorem 4.3.2 (a), p. 178]{edw-bo}), we infer $\langle\mu_{\operatorname{Re},+},\varphi\rangle\leq \sup_{\phi\in\mathcal{K}(\RR^n),\, |\phi|\leq \varphi}|\langle \mu,\phi\rangle|$. Analogous estimates hold for the rest of the measures in the decomposition of $\mu$. Since $B':=\{\phi\in\mathcal{K}(\RR^n)\,|\, |\phi|\leq \varphi\}$ is a bounded subset of $\mathcal{K}_{\supp\varphi}(\RR^n)$, we deduce $\sup_{f\in B}\|f*\mu\|_{L^1(K')}\leq 4(\sup_{f\in B}\|f\|_{L^1(\RR^n)})\sup_{\phi\in B'}|\langle\mu,\phi\rangle|$ which completes the proof for the continuity of the map in $(a)$.\\
\indent Concerning the injectivity, again, we only show it for the map in $(a)$ as the rest are analogous. Let $f*\mu=0$ for all $f\in L^1_{\comp}(\RR^n)$. This implies that $\mu$ is zero on $\operatorname{span}(\mathcal{K}(\RR^n)*\mathcal{K}(\RR^n))$ and hence $\mu=0$ since $\operatorname{span}(\mathcal{K}(\RR^n)*\mathcal{K}(\RR^n))$ is dense in $\mathcal{K}(\RR^n)$.\\
\indent Next, we address the rest of the claims in $(a)$. Let $Q\in\mathcal{L}(L^1_{\comp}(\RR^n), L^1_{\loc}(\RR^n))$ commutes with all translations. We employ a standard idea (see the proof of \cite[Theorem 2.5.8, p. 153]{Grafakos}) to show that it is a convolution with a Radon measure. In view of \cite[Theorem, p. 332]{edw-bo}, there is $u\in\DD'(\RR^n)$ such that $Q\varphi=\varphi*u$, $\varphi\in\DD(\RR^n)$. Pick nonnegative $\chi\in\DD(\RR^n)$ satisfying $\supp\chi\subseteq B_1(0)$ and $\int_{\RR^n}\chi(x)dx=1$ and set $\chi_k(x)=k^n\chi(kx)$, $x\in\RR^n$, $k\in\ZZ_+$. Then $\{\chi_k*u\}_{k\in\ZZ_+}$ is a bounded subset of $L^1_{\loc}(\RR^n)$ and hence bounded in $\mathcal{K}'(\RR^n)$ too. As $\mathcal{K}(\RR^n)$ is barrelled, $\{\chi_k*u\}_{k\in\ZZ_+}$ is equicontinuous in $\mathcal{K}'(\RR^n)$ and the Banach-Alaoglu theorem \cite[Corollary, p. 84]{Sch} implies that it is weakly relatively compact in $\mathcal{K}'(\RR^n)$. Furthermore, its closure in the weak topology is metrisable; see \cite[Theorem 1.7, p. 128]{Sch}. Thus, there exists a subsequence $(\chi_{k_l}*u)_{l\in\ZZ_+}$ and $\mu\in\mathcal{K}'(\RR^n)$ such that $\langle \chi_{k_l}*u,\varphi\rangle\rightarrow\langle \mu,\varphi\rangle$, $\varphi\in\DD(\RR^n)$. Since $\chi_k*u\rightarrow u$ in $\DD'(\RR^n)$, we conclude $u=\mu$. It remains to show that the map is an open map onto the image. Let $B$ be a bounded subset of $\mathcal{K}(\RR^n)$. There is a compact set $K\subseteq \RR^n$ such that $B$ is a bounded subset of $\mathcal{K}_K(\RR^n)$. Employing the above notation, for $\mu\in\mathcal{K}'(\RR^n)$ we have
$$
\sup_{\psi\in B}|\langle \mu,\psi\rangle|\leq \sup_{\psi\in B}\sup_{k\in\ZZ_+}|\langle\chi_k*\mu,\psi\rangle|\leq (\sup_{\psi\in B}\|\psi\|_{L^{\infty}(\RR^n)})\sup_{k\in\ZZ_+}\|\chi_k*\mu\|_{L^1(K)}
$$
which completes the proof since $\sup_{k\in\ZZ_+}\|\chi_k*\mu\|_{L^1(K)}$ is a continuous seminorm on $\mathcal{L}_b(L^1_{\comp}(\RR^n),L^1_{\loc}(\RR^n))$ of the map $f\mapsto f*\mu$.\\
\indent We now address $(b)$. Let $Q\in\mathcal{L}(\mathcal{K}(\RR^n), \mathcal{C}(\RR^n))$ commutes with all translations. Whence there is $u\in\DD'(\RR^n)$ such that $Q\varphi=\varphi*u$, $\varphi\in\DD(\RR^n)$. With $\{\chi_k\}_{k\in\ZZ_+}$ as above, we are going to show that $\{\chi_k*u\}_{k\in\ZZ_+}$ is an equicontinuous subset of $\mathcal{K}'(\RR^n)$. Since $\mathcal{K}(\RR^n)$ is barrelled, it suffices to show it is weakly bounded in $\mathcal{K}'(\RR^n)$. Let $\varphi\in\mathcal{K}(\RR^n)$ and pick a sequence $(\varphi_m)_{m\in\ZZ_+}$ in $\DD_K(\RR^n)$ with $K\supseteq\supp\varphi$ such that $\varphi_m\rightarrow \varphi$ in $\mathcal{K}_K(\RR^n)$. Then
$$
\sup_{k\in\ZZ_+}|\langle \chi_k*u,\varphi\rangle|\leq \sup_{k,m\in\ZZ_+}|\langle \check{\varphi}_m*u,\check{\chi}_k\rangle|\leq \|\chi\|_{L^1(\RR^n)}\sup_{m\in\ZZ_+}\|\check{\varphi}_m*u\|_{L^{\infty}(B_1(0))}<\infty
$$
which implies the weak boundedness of $\{\chi_k*u\}_{k\in\ZZ_+}$. Now, analogously as above, we infer that $u\in\mathcal{K}'(\RR^n)$. To show that the map is an open map onto the image, with $B$ as above and $\mu\in\mathcal{K}'(\RR^n)$, we have
$$
\sup_{\psi\in B}|\langle \mu,\psi\rangle|\leq \sup_{\psi\in B}\sup_{k\in\ZZ_+}|\langle\check{\psi}*\mu,\check{\chi}_k\rangle|\leq \|\chi\|_{L^1(\RR^n)} \sup_{\psi\in B}\|\check{\psi}*\mu\|_{L^{\infty}(B_1(0))}
$$
which completes the proof of $(b)$ since $\sup_{\psi\in B}\|\check{\psi}*\mu\|_{L^{\infty}(B_1(0))}$ is a continuous seminorm on $\mathcal{L}_b(\mathcal{K}(\RR^n),\mathcal{C}(\RR^n))$ of the map $\varphi\mapsto \varphi*\mu$. The proof of $(c)$ is analogous to the proof of $(b)$ and we omit it.
\end{proof}

As a consequence of Proposition \ref{pro-for-not-multfunradcon} and Proposition \ref{spa-ofm-for-lolicont}, we deduce the main result.

\begin{theorem}\label{mai-the-ofa-insl}
Let $V\subseteq \RR^n$ be a closed cone satisfying $V\backslash\{0\}\neq\emptyset$ and $V':=\operatorname{int}V^*\neq\emptyset$. Let $\psi\in\mathcal{C}^{\infty}(\RR^n)$ be such that $(x,\xi)\mapsto \psi(\xi)$ belongs to $S^0_{1,0}(\RR^{2n})$. Assume that $\operatorname{Re}\psi$ is positively homogeneous of order $0$ and $\supp(\operatorname{Re}\psi)\subseteq V$. If the continuous function $\varphi_1:\Ss^{n-1}\rightarrow \RR$, $\varphi_1(\omega):=\lim_{r\rightarrow\infty}\operatorname{Re}\psi(r\omega)$, is such that $\varphi_1$ or $-\varphi_1$ satisfies \eqref{con-for-exi-offournotmeas} then there are $f\in L^1_{\comp}(\RR^n)$, $g\in L^{\infty}_{\comp}(\RR^n)$ and $\varphi\in\mathcal{K}(\RR^n)$ such that $\psi(D)f\not\in L^1_{\loc}(\RR^n)$, $\psi(D) g\not\in L^{\infty}_{\loc}(\RR^n)$ and $\psi(D)\varphi\not\in\mathcal{C}(\RR^n)$.\\
\indent The same conclusion holds if these assumptions are satisfied by $\operatorname{Im}\psi$ instead of $\operatorname{Re}\psi$.
\end{theorem}

\begin{proof} Let $(X,Y)$ be either one of the pairs $(L^1_{\comp}(\RR^n),L^1_{\loc}(\RR^n))$, $(L^{\infty}_{\comp}(\RR^n),L^{\infty}_{\loc}(\RR^n))$ or $(\mathcal{K}(\RR^n),\mathcal{C}(\RR^n))$. The map $\psi(D): X\rightarrow \DD'(\RR^n)$ is continuous since it is the composition of the continuous maps $X\xrightarrow{\operatorname{Id}} \SSS'(\RR^n)\xrightarrow{\psi(D)}\SSS'(\RR^n)\xrightarrow{\operatorname{Id}}\DD'(\RR^n)$. Assume that $\psi(D)(X)\subseteq Y$. The above implies that $\psi(D):X\rightarrow Y$ is well-defined and with closed graph. The De Wilde closed graph theorem \cite[Theorem 2, p. 57]{kothe2} yields that $\psi(D):X\rightarrow Y$ is continuous ($X$ is ultrabornological in view of \cite[Theorem 7, p. 72]{kothe2}, while \cite[Theorem 4, p. 55]{kothe2} verifies that $Y$ is a strictly webbed space). Consequently, Proposition \ref{spa-ofm-for-lolicont} implies that there is $\mu\in\mathcal{K}'(\RR^n)$ such that $\psi(D)\varphi=\varphi*\mu$, $\varphi\in\DD(\RR^n)$. For $\varphi,\chi\in\DD(\RR^n)$, we have
$$
\langle\mathcal{F}^{-1}\psi,\varphi*\chi\rangle=\langle(\mathcal{F}^{-1}\psi)*\check{\varphi},\chi\rangle=\langle \psi(D)\check{\varphi},\chi\rangle=\langle \check{\varphi}*\mu,\chi\rangle=\langle\mu,\varphi*\chi\rangle.
$$
As $\operatorname{span}(\DD(\RR^n)*\DD(\RR^n))$ is dense in $\DD(\RR^n)$, we infer $\mathcal{F}^{-1}\psi=\mu$. This is in contradiction with Proposition \ref{pro-for-not-multfunradcon}, since $\mathcal{F}\psi=(2\pi)^n(\mathcal{F}^{-1}\psi)\check{}\in\mathcal{K}'(\RR^n)$, and the proof is complete.
\end{proof}

\begin{remark}
As before, it suffices to verify that $\varphi_1$ or $-\varphi_1$ satisfy \eqref{con-for-exi-offournotmeas11}. If $\varphi_1$ is nonnegative and not identically equal to zero, then it always satisfies \eqref{con-for-exi-offournotmeas11}.
\end{remark}

\section{Wave front sets for a class of Banach spaces}

In this section we construct wave front sets that will measure the microlocal regularity with respect to a large class of Banach spaces of distributions. As special cases, they will include most of the standard wave front sets considered in the literature like the Sobolev wave front set \cite{hor1,pil-pra11}, the Besov wave front set \cite{dap-r-scl}, wave front sets with respect to Banach space having solid Fourier side \cite{cor-joh-tof1,pbb}, etc (see also \cite{cor-joh-tof2} for global wave front sets with respect to Banach spaces of distributions). Furthermore, as a consequence of the results in the previous section, we show that if we require the wave front set to behave in a natural way, then it is not possible to construct it such that it measures the $L^1$ or $L^{\infty}$ microlocal regularity, neither the microlocal continuity of a distribution. To make things precise, we recall the following definition and facts from \cite{dpv}. The Banach spaces $F$ with norm $\|\cdot\|_F$ is said to be a \textit{translation-invariant Banach space of distributions}, or (TIB) space for short, if it satisfies the following assumptions:
\begin{itemize}
\item[$(I)$] $\SSS(\RR^n)\subseteq F\subseteq \SSS'(\RR^n)$ with continuous and dense inclusions;
\item[$(II)$] $T_x\in\mathcal{L}(F)$, $x\in\RR^n$;
\item[$(III)$] there are $C,\tau>0$ such that $\omega_F(x):=\|T_x\|_{\mathcal{L}_b(F)}\leq C\langle x\rangle^{\tau}$.
\end{itemize}
One can show that $\omega_F$ is submultiplicative and Borel measurable. Furthermore, for each $f\in F$, the map $\RR^n\rightarrow F$, $x\mapsto T_x f$, is continuous. The convolution $*:\SSS(\RR^n)\times\SSS(\RR^n)\rightarrow \SSS(\RR^n)$ uniquely extends to a continuous bilinear mapping $*:L^1_{\omega_F}(\RR^n)\times F\rightarrow F$ and $F$ becomes a Banach module over the Beurling algebra $L^1_{\omega_F}(\RR^n)$:
\begin{equation}\label{ban-mod-int-repbochs}
\|g*f\|_F\leq \|g\|_{L^1_{\omega_F}(\RR^n)}\|f\|_F,\quad \mbox{where}\quad g*f=\int_{\RR^n}g(x)T_xf dx,
\end{equation}
with the last being a Bochner integral of the $F$-valued function $x\mapsto g(x)T_x f$.\\
\indent The Banach space $F$ will be called a \textit{dual translation-invariant Banach space of distributions}, or (DTIB) space for short, if $F$ is the dual of a (TIB) space $F_0$. The space $F$ always satisfies the continuous inclusions $\SSS(\RR^n)\subseteq F\subseteq \SSS'(\RR^n)$ but they may fail to be dense (e.g. $F=L^{\infty}(\RR^n)$). Furthermore, $F$ always satisfies $(II)$ and $(III)$ and $\omega_F=\check{\omega}_{F_0}$. However, for $f\in F$ fixed, the map $\RR^n\rightarrow F$, $x\mapsto T_x f$, is only (in general) continuous with respect to the weak* topology on $F$. One defines the convolution $*:L^1_{\omega_F}(\RR^n)\times F\rightarrow F$ by duality, i.e. $\langle g*f,f_0\rangle:=\langle f,\check{g}*f_0\rangle$, $f\in F$, $f_0\in F_0$, $g\in L^1_{\omega_F}(\RR^n)$. Then $F$ becomes a Banach module over $L^1_{\omega_F}(\RR^n)$ and \eqref{ban-mod-int-repbochs} holds true but the integral should be interpreted as a Pettis integral taken in the weak* topology of $F$.\\
\indent Let now $E$ be a Banach space of distributions with norm $\|\cdot\|_E$ which satisfies the following conditions:
\begin{itemize}
\item[$(i)$] $\SSS(\RR^n)\subseteq E\subseteq \SSS'(\RR^n)$ with continuous inclusions;
\item[$(ii)$] the Banach space $\mathcal{F}E:=\{f\in\SSS'(\RR^n)\,|\, \mathcal{F}^{-1}f\in E\}$ with norm $\|f\|_{\mathcal{F}E}:=\|\mathcal{F}^{-1}f\|_E$ is a (TIB) or a (DTIB) space;
\item[$(iii)$] $\psi f\in\mathcal{F}E$ for all $\psi\in\SSS(\RR^n)$, $f\in\mathcal{F}E$.
\end{itemize}
For such $E$, it also holds that
\begin{itemize}
\item[$(iii)^*$] the bilinear map $\SSS(\RR^n)\times E\rightarrow E$, $(\varphi,e)\mapsto \varphi e$, is continuous;
\item[$(iii)^{**}$] the bilinear map $\SSS(\RR^n)\times \mathcal{F}E\rightarrow \mathcal{F}E$, $(\psi,f)\mapsto \psi f$, is continuous.
\end{itemize}
To verify $(iii)^*$, for $\varphi\in \SSS(\RR^n)$ and $e\in E$, $(ii)$ and \eqref{ban-mod-int-repbochs} give $\varphi e=(2\pi)^{-n}\mathcal{F}^{-1}(\mathcal{F}\varphi*\mathcal{F}e)\in E$ and also show the continuity of the map. To prove $(iii)^{**}$, it suffices to show it is separately continuous in view of \cite[Theorem 1, p. 158]{kothe2}. The separate continuity follows by a classical application of the closed graph theorem.\\
\indent Given such $E$, we define the following two auxiliary spaces
$$
E_{\loc}:=\{f\in \DD'(\RR^n)\,|\, \varphi f\in E,\, \forall \varphi\in\DD(\RR^n)\},\quad E_{\comp}:=\{e\in E\,|\, \supp e\,\, \mbox{is compact}\}.
$$
(One can equip $E_{\loc}$ and $E_{\comp}$ with natural Fr\'echet and strict $(LB)$-space topologies respectively, but we will not need these facts.) We wish to define a wave front set which will keep track of the directions where a given distribution does not behave like a $E_{\loc}$ function. Given $u\in E_{\loc}$ and $x_0\in\RR^n$, for $\varphi\in\DD(\RR^n)$ satisfying $\varphi(x_0)\neq0$, we always have $\varphi u\in E_{\comp}$ and $\mathcal{F}(\varphi u)\in \mathcal{F}E$. In order for the wave front set to behave in a natural way, in every direction in the frequency space, $\varphi u$ has to behave like an element of $E$, i.e. for each $\xi\in\RR^n\backslash\{0\}$ and open cone $V\ni\xi_0$ and every function $\psi$ with support in $V$ and equaling a positive constant in a cone neighbourhood of $\xi_0$ away from the origin it should hold $\psi \mathcal{F}(\varphi u)\in \mathcal{F}E$. This means that $\mathcal{F}E$ has to have multipliers of this form. Of course, assuming that the characteristic functions of the open cones are multipliers for $\mathcal{F}E$ is too restrictive: when $p\in(2,\infty)$, $\mathcal{F}L^p$ contains distributions of positive order (i.e. which are not measures, see \cite[Corollary 1.5]{hor-p1}) and the multiplication with characteristic functions is meaningless. The idea is to allow multiplication of $\mathcal{F}E$ with a very restrictive class of functions with supports in open cones and equaling to positive constants in subcones away from the origin. (The more restrictive the class of multipliers is, the more spaces $E$ will have it as Fourier multipliers!) From now, we assume that $E$ has the following functions as Fourier multipliers:
\begin{itemize}
\item[$(iv)$] $\psi f\in\mathcal{F}E$, $f\in\mathcal{F}E$, for all $\psi\in\mathcal{C}^{\infty}(\RR^n)$ which are positively homogeneous of order $0$.
\end{itemize}
As before, the closed graph theorem implies that the map $\mathcal{F}E\rightarrow\mathcal{F}E$, $f\mapsto \psi f$, is continuous. We point out that we do not impose any solidity assumptions neither on $E$ nor on $\mathcal{F}E$. Of course, if $\mathcal{F}E$ is solid then it clearly satisfies $(iv)$. The spaces $L^p(\RR^n)$, $1<p<\infty$, and the corresponding Sobolev spaces $W^{r,p}(\RR^n)$, $r\in\RR$, $1<p<\infty$, satisfy $(i)$, $(ii)$, $(iii)$ and $(iv)$ (cf. \cite[Theorem 6.2.7, p. 446]{Grafakos}). The Besov spaces $B^s_{p,q}(\RR^n)$, $s\in\RR$, $p,q\in[1,\infty)$ or $p=q=\infty$, also satisfy $(i)$, $(ii)$, $(iii)$ and $(iv)$ (see \cite[Theorem 2.17, p. 257, and Corollary 5.2, p. 608]{sawano} and \cite[Theorem, p. 140]{triebel}); consequently, the H\"older spaces of non-integer order $s>0$ satisfy these conditions too since they coincide with $B^s_{\infty,\infty}(\RR^n)$. It will be convenient to introduce the following terminology.

\begin{definition}
Let $V$ be a closed cone in $\RR^n$ and $V'$ an open cone satisfying $V\backslash\{0\} \subseteq V'$. The smooth nonnegative function $\psi$ is said to be a \textit{smooth cut-off} for the pair $(V,V')$ if $\supp\psi\subseteq V'$, $\psi$ is positively homogeneous of order $0$ outside of $B_R(0)$ for some $R>0$ and $\psi$ equals a positive constant on $V\backslash B_R(0)$.
\end{definition}

\begin{remark}
Given $V$ and $V'$ as in the definition, there always exists a smooth cut-off $\psi$ for $(V,V')$. To see this, take $\varphi\in\DD(\RR^n)$ such that $0\leq \varphi\leq 1$, $\supp\varphi\subseteq V'$ and $\varphi=1$ on a neighbourhood of $V\cap\mathbb{S}^{n-1}$. Choose $\chi\in\DD(\RR^n)$ such that $0\leq\chi\leq 1$, $\supp\chi\subseteq B_1(0)$ and $\chi=1$ on $B_{1/2}(0)$. The function $\psi(x):=(1-\chi(x))\varphi(x/|x|)$ is a smooth cut-off for $(V,V')$.
\end{remark}

\begin{remark}
If $\psi_j$ is a smooth cut-off for $(V_j,V'_j)$, $j=1,2$, then $\psi_1\psi_2$ is a smooth cut-off for $(V_1\cap V_2,V'_1\cap V'_2)$.
\end{remark}

Given such $E$, we define the wave front set of a distribution with respect to $E$ as follows. We first define the set $\Sigma^E(u)$ for $u\in\EE'(\RR^n)$. The point $\xi\in\RR^n\backslash\{0\}$ does not belong to $\Sigma^E(u)$ if there is an open cone $V$ containing $\xi$ and a smooth cut-off $\psi$ for $(\overline{V},\RR^n)$ such that $\psi\mathcal{F}u\in\mathcal{F}E$. Clearly, $\Sigma^E(u)$ is a closed cone in $\RR^n\backslash\{0\}$.

\begin{lemma}
For $\varphi\in\DD(\RR^n)$ and $u\in\EE'(\RR^n)$, it holds that $\Sigma^E(\varphi u)\subseteq \Sigma^E(u)$. Furthermore, $\Sigma^E(u)=\emptyset$ if and only if $u\in E_{\comp}$.
\end{lemma}

\begin{proof} The compactness of the unit sphere together with a classical argument yields the validity of the second part. To show the first part, let $\xi_0\not\in \Sigma^E(u)$. There exist an open cone $V\ni\xi_0$ and smooth cut-off $\psi$ for $(\overline{V},\RR^n)$ such that $\psi\mathcal{F}u\in \mathcal{F}E$. Without loss in generality, we can assume that $\RR^n\backslash\overline{V}\neq\emptyset$. Denote by $c_0>0$ the constant that $\psi$ equals to in $V$ away from the origin. Choose open cones $V_1$ and $V_2$ such that $\xi_0\in V_1$, $\overline{V_1}\backslash\{0\}\subseteq V_2$ and $\overline{V_2}\backslash \{0\}\subseteq V$ and pick a smooth cut-off $\chi$ for $(\overline{V_1},V_2)$. Write
$$
\chi(\xi)\mathcal{F}(\varphi u)(\xi)=(2\pi)^{-n}c_0^{-1}(I_1(\xi)+I_2(\xi)),
$$
with
\begin{align}
I_1(\xi)&:=\chi(\xi)\int_{\RR^n}\mathcal{F}\varphi(\eta)\psi(\xi-\eta)\mathcal{F}u(\xi-\eta)d\eta,\label{fun-whi-bel-toeinfrs}\\
I_2(\xi)&:=\chi(\xi)\int_{\RR^n}\mathcal{F}\varphi(\eta)(c_0-\psi(\xi-\eta))\mathcal{F}u(\xi-\eta)d\eta.\nonumber
\end{align}
We are going to show that $I_1\in \mathcal{F}E$ and $I_2\in\SSS(\RR^n)$ which will complete the proof. To prove that $I_2\in \SSS(\RR^n)$, notice that there is $k\in\ZZ_+$ such that $\|\langle\cdot\rangle^{-k}\partial^{\alpha}\mathcal{F}u\|_{L^{\infty}(\RR^n)}<\infty$, for all $\alpha\in\NN^n$. Write
\begin{equation}\label{ine-for-fir-parforintsest}
\langle \xi\rangle^l|\partial^{\alpha}I_2(\xi)|\leq C'\sum_{\beta'+\beta''\leq\alpha} \int_{\RR^n}\langle \eta\rangle^l|\mathcal{F}\varphi(\eta)||\partial^{\beta'}\chi(\xi)||\partial^{\beta''}(c_0-\psi(\xi-\eta))|\langle \xi-\eta\rangle^{k+l}d\eta.
\end{equation}
There is $0<\varepsilon<1/2$ such that $|\omega-\omega'|\geq 2\varepsilon$, $\omega\in \overline{V_2}\cap \mathbb{S}^{n-1}$, $\omega'\in (\RR^n\backslash V)\cap \mathbb{S}^{n-1}$. For $\xi\in \overline{V_2}\backslash\{0\}$ and $\xi'\in\RR^n\backslash (V\cup\{0\})$, we infer
$$
\left|\frac{\xi'}{|\xi'|}-\frac{\xi}{|\xi'|}\right|\geq \left|\frac{\xi'}{|\xi'|}-\frac{\xi}{|\xi|}\right|-\left|\frac{\xi}{|\xi|}-\frac{\xi}{|\xi'|}\right|\geq 2\varepsilon-\frac{||\xi'|-|\xi||}{|\xi'|}\geq 2\varepsilon-\frac{|\xi'-\xi|}{|\xi'|}
$$
from what we deduce that $|\xi'-\xi|\geq \varepsilon|\xi'|$. We only need to estimate the integrand when $\xi\in V_2$ and $\xi-\eta\in\RR^n\backslash (V\cup\{0\})$, for otherwise the right-hand side in \eqref{ine-for-fir-parforintsest} is uniformly bounded in $\xi$. We apply the above inequality with $\xi'=\xi-\eta$ to deduce that $|\eta|\geq \varepsilon|\xi-\eta|$ which immediately gives that the right-hand side in \eqref{ine-for-fir-parforintsest} is uniformly bounded in $\xi$. This implies $I_2\in\SSS(\RR^n)$. To show that $I_1\in\mathcal{F}E$, notice that the function given by the integral in \eqref{fun-whi-bel-toeinfrs} is exactly
$$
\int_{\RR^n}\mathcal{F}\varphi(\eta)T_{\eta}(\psi\mathcal{F}u) d\eta\in \mathcal{F}E,
$$
where the integral should be interpreted as a Bochner integral if $\mathcal{F}E$ is a (TIB) and as a Pettis integral with respect to the weak* topology in $\mathcal{F}E$ if the latter is a (DTIB). This yields $I_1\in\mathcal{F}E$ and the proof is complete.
\end{proof}

If $O$ is an open set in $\RR^n$ and $u\in\DD'(O)$, for each $x\in O$ we define
$$
\Sigma_x^E(u):=\bigcap_{\varphi\in\DD(O),\, \varphi(x)\neq 0}\Sigma^E(\varphi u).
$$
Of course, $\Sigma^E_x(u)$ is a closed cone in $\RR^n\backslash\{0\}$. Furthermore, a standard compactness argument (cf. \cite[p. 254]{hor}) together with the above lemma yield that if $V$ is an open cone containing $\Sigma^E_x(u)$, then there is an open neighbourhood $O'\subseteq O$ of $x$ having a compact closure in $O$ such that $\Sigma^E(\varphi u)\subseteq V$ for all $\varphi\in\DD(O')$.

\begin{definition}
Let $O$ be an open set in $\RR^n$. For every $u\in\DD'(O)$ we define the $E$-wave front set of $u$ by:
$$
WF^E(U):=\{(x,\xi)\in O\times (\RR^n\backslash\{0\})\,|\, \xi\in \Sigma^E_x(u)\}.
$$
\end{definition}

Of course, $WF^E(u)$ is a closed conic subset of $O\times(\RR^n\backslash\{0\})$. For $u\in \DD'(O)$, we define the $E$-singular support of $u$ as the complement of the set of points where $u$ is behaving as an $E_{\loc}$ distribution. To be precise:
$$
\operatorname{sing}\supp_E(u):=O\backslash\{x\in O\,|\, \exists \varphi\in\DD(O)\,\, \mbox{satisfying}\,\, \varphi(x)\neq0\,\, \mbox{such that}\,\, \varphi u\in E\}.
$$
Clearly, $\operatorname{sing}\supp_E(u)$ is closed in $O$. In view of the above and repeating the proof of \cite[Proposition 8.1.3, p. 254]{hor} verbatim, one shows the following result.

\begin{proposition}
For each $u\in\DD'(O)$, $\operatorname{pr}_1(WF^E(u))=\operatorname{sing}\supp_E(u)$. Furthermore, for every $v\in\EE'(\RR^n)$, $\operatorname{pr}_2(WF^E(v))=\Sigma^E(v)$.
\end{proposition}

Finally, in view of Theorem \ref{mai-the-ofa-insl}, we immediately deduce the following result.

\begin{corollary}
Let $\psi$ be a smooth cut-off for $(V_0,V)$. If $V_0\backslash\{0\}\neq\emptyset$ and $\operatorname{int}V^*\neq\emptyset$ then there are $f\in L^1_{\comp}(\RR^n)$, $g\in L^{\infty}_{\comp}(\RR^n)$ and $\varphi\in\mathcal{K}(\RR^n)$ such that $\psi(D)f\not\in L^1_{\loc}(\RR^n)$, $\psi(D) g\not\in L^{\infty}_{\loc}(\RR^n)$ and $\psi(D)\varphi\not\in\mathcal{C}(\RR^n)$.
\end{corollary}

Hence, if one wants for every smooth cut-off $\psi$ and every $u\in E_{\comp}(\RR^n)$ to hold $\psi\mathcal{F}u\in \mathcal{F}E$, then one can never define a wave front for $E=L^1(\RR^n)$, $E=L^{\infty}(\RR^n)$ and $E=\mathcal{C}_0(\RR^n)$.


\begin{thebibliography}{999}

\bibitem{a-h} J. Alvarez, J. Hounie, \textit{Estimates for the kernel and continuity properties of pseudo-differential operators}, Ark. Mat. 28(1) (1990), 1-22.

\bibitem{bourbaki} N. Bourbaki, \textit{Integration I}, Springer Berlin, Heidelberg, 2004.

\bibitem{cor-joh-tof1} S. Coriasco, K. Johansson, J. Toft, \textit{Local wave-front sets of Banach and Fr\'echet types, and pseudo-differential operators}, Monatsh. Math. 169(3-4) (2013), 285-316.

\bibitem{cor-joh-tof2} S. Coriasco, K. Johansson, J. Toft, \textit{Global wave-front sets of Banach, Fréchet and modulation space types, and pseudo-differential operators}, J. Differ. Equations 254(8) (2013), 3228-3258.

\bibitem{dap-r-scl} C. Dappiaggi, P. Rinaldi, F. Sclavi, \textit{Besov wavefront set}, Anal. Math. Phys. 13(6) (2023), Paper No. 95.

\bibitem{dpv} P. Dimovski, S. Pilipovi\'c, J. Vindas, \textit{New distribution spaces associated to translation-invariant Banach spaces}, Monatsh. Math. 177(4) (2015), 495-515.

\bibitem{pbb} P. Dimovski, B. Prangoski, \textit{Wave Front Sets with Respect to Banach Spaces of Ultradistributions. Characterisation via the Short-Time Fourier Transform}, Filomat 33(18) (2019), 5829-5836.

\bibitem{edw-bo} R. E. Edwards, \textit{Functional analysis. Theory and applications}, Holt, Rinehart and Winston, New York, 1965.

\bibitem{fefferman} C. Fefferman, \textit{$L^p$ bounds for pseudo-differential operators}, Isr. J. Math. 14 (1973), 413-417.

\bibitem{Grafakos} L. Grafakos, \emph{Classical Fourier analysis}, third edition, Springer, New York, 2014.

\bibitem{g-z} J. Guo, X. Zhu, \textit{$L^p$ boundedness of pseudo-differential operators with symbols in $S^{n(\rho-1)/2}_{\rho,1}$}, J. Math. Anal. Appl. 539(2) (2024), Article ID 128538.

\bibitem{hor-p1} L. H\"ormander, \textit{Estimates for translation invariant operators in $L^p$ spaces}, Acta Math. 104 (1960), 93-140.

\bibitem{hor1} L. H\"ormander, \textit{Lectures on Nonlinear Hyperbolic Differential Equations}, Springer, 1997.

\bibitem{hor} L. H\"ormander, \textit{The analysis of linear partial differential operators I. Distribution theory and fourier analysis}, Springer, 2003.

\bibitem{hor2} L. H\"ormander, \textit{The analysis of linear partial differential operators III. Pseudo-differential operators}, Springer, 2007.

\bibitem{illner} R. Illner, \textit{A class of $L^p$-bounded pseudo-differential operators}, Proc. Am. Math. Soc. 51 (1975), 347-355.

\bibitem{kothe2} G.~K\"{o}the, \textit{Topological vector spaces II}, Springer, New York, 1979.

\bibitem{lizorkin} P. I. Lizorkin, \textit{On a theorem of Marcinkiewicz type for $H$-valued functions. A continual form of the Paley-Littlewood theorem}, (in Russian) Math. USSR, Sb. 16 (1972), 237-243.

\bibitem{mikhlin} S. G. Mikhlin, \textit{On the multipliers of Fourier integrals}, (in Russian) Dokl. Akad. Nauk. SSSR 109 (1956), 701-703.

\bibitem{pil-pra11} S. Pilipovi\'c, B. Prangoski, \textit{Spaces of distributions with Sobolev wave front in a fixed conic set: compactness, pullback by smooth maps and the compensated compactness theorem}, preprint, arXiv:2408.10741.

\bibitem{rodino} L. Rodino, \textit{On the boundedness of pseudo differential operators in the class $L^m_{\rho,1}$}, Proc. Am. Math. Soc. 58 (1976) 211–215.

\bibitem{triebel} H. Triebel, \textit{Theory of function spaces}, Springer, Basel, 2010

\bibitem{sawano} Y. Sawano, \textit{Theory of Besov spaces}, Springer, Singapore, 2018.

\bibitem{Sch} H. H. Schaefer, \textit{Topological Vector Spaces}, Springer-Verlag, New York Heidelberg Berlin, 1970.

\bibitem{stein} E. M. Stein, \textit{Harmonic analysis: Real-variable methods, orthogonality, and oscillatory integrals}, Princeton University Press, Princeton, New Jersey, 1993.

\bibitem{taylor} M. E. Taylor, \textit{Pseudodifferential operators}, Princeton University Press, Princeton, New Jersey, 1981.

\bibitem{wan} G. Wang, \textit{Sharp function and weighted $L^p$ estimates for pseudo-differential operators with symbols in general H\"ormander classes}, preprint, arXiv:2206.09825.


\end{thebibliography}
\end{document}